\documentclass{article}
\usepackage{amsmath,amssymb,amsfonts,amsthm}
\usepackage{graphicx}
\usepackage{layout}
\usepackage[latin1]{inputenc}
\usepackage[french,english]{babel}
\usepackage[french]{babel}
\usepackage{cancel}

\newtheorem{theorem}{Theorem}
\newtheorem{definition}{Definition}

\newtheorem{corollary}{Corollary}
\newtheorem{proposition}{Proposition}
\newtheorem{lemma}{Lemma}

\newtheorem{notation}{Notation}

\newtheorem{example}{Example}
\newtheorem{remark}{Remark}    

\usepackage{fancyhdr}

\usepackage{setspace}

\renewcommand{\headrulewidth}{0,5pt}
\newcommand{\vertiii}[1]{{\left\vert\kern-0.25ex\left\vert\kern-0.25ex\left\vert #1 
    \right\vert\kern-0.25ex\right\vert\kern-0.25ex\right\vert}} 

\title{On a class of  Schauder frames in Banach spaces}
\author{{\footnotesize Samir Kabbaj, Rafik Karkri and Zoubeir Hicham}  }

\date{\today}

\begin{document}
\maketitle 

\begin{abstract}
In this paper, we give a characterization and a some  properties of a besselian sequences, 
which allows us to build some examples of a besselian Schauder frames. Also
for   a  reflexive Banach spaces (with 
 a besselian Schauder frames) we give some characterizations.
\end{abstract}
{\footnotesize $\textbf{MSC2020-Mathematics Subject Classification System}$. 
46B04, 46B10, 46B15, 46B25, 46B45.}\\
{\footnotesize $\textbf{Keywords and phrases:}
$ Schauder frame, Frame, Weakly sequentially complete Banach space, Besselian sequence, 
 Besselian  Schauder frame.}

\renewcommand{\headrulewidth}{0pt}  
\fancyhead[C]{{\scriptsize Samir Kabbaj, Rafik Karkri  and Hicham Zoubeir }} 

\maketitle

\section{Introduction}

In $1946,$ Gabor \cite{gab} formulated a fundamental approach 
to signal decomposition in terms of
elementary signals. In $1952,$ Duffin and Schaeffer  
\cite{duf} abstracted the fundamental notion of Gabor for
  introduced  the notion of frames for Hilbert spaces as a generalization of basis. 
 A frames for a Banach spaces was 
 introduced in $(1991)$ by Gr\"{o}chenig \cite{gro}
 to extend the definition of frames to that of general separable Banach spaces. 
 In light of the works of  Cassaza, Han and Larson  \cite
{cas} $(1999)$,    Han and Larson \cite{lar2} $(2000)$.
 Cassaza \cite{cas.2008} $(2008)$  introduced the notion of
Schauder frame of a given  Banach space. In 2021 Karkri and Zoubeir
  \cite{karkri.zoubeir.2021} introduced the notion of
besselian Schauder frames of Banach spaces and obtained some results  on
universal  spaces and complemented subspaces for Banach spaces with Schauder frames 
or with besselian Schauder frames.
 For more generalizations  of  frame notion one can refer to 
 \cite{cas2, Naroei.Nazari, Kim.Lim}

In this paper, we give a characterization and a some  properties of a besselian sequences, 
which allows us to build some examples of a besselian Schauder frames. In the other hand,
 for   a  reflexive Banach spaces (with 
 a besselian Schauder frames) we give some characterizations.
\section{Main definitions, principal notations and useful remarks}
Let $\left( X,\left \Vert .\right \Vert _{X}\right) $ be a Banach space on $%
\mathbb{K}\in \left \{ \mathbb{R},\mathbb{C}\right \} $ and $X^{\ast },$ $%
X^{\ast \ast }$and $X^{\ast \ast \ast }$ respectively its first, second and
third topological duals.

\begin{enumerate}
\item We denote by $\mathbb{B}_{X}$ the closed unit ball of $X:$%
\begin{equation*}
\mathbb{B}_{X}:=\left \{ x\in X:\left \Vert x\right \Vert _{X}\leq 1\right \}
\end{equation*}

\item We denote by $l^{1}(\mathbb{K)}$ the $\mathbb{K}$-vector space of
sequences $\lambda :=\left( \lambda _{n}\right) _{n\in 
\mathbb{N}
^{\ast }}$ such that $\lambda _{n}\in \mathbb{K}$ for each $n\in 
\mathbb{N}
^{\ast }$ and $\sum_{n=1}^{+\infty }\left \vert \lambda _{n}\right \vert
<+\infty .\ $It is a classical result that $l^{1}(\mathbb{K)}$ is a Banach
space for the norm : 
\begin{equation*}
\begin{array}{cccc}
\left \Vert .\right \Vert _{l^{1}(\mathbb{K)}}: & l^{1}(\mathbb{K)} & 
\rightarrow & 
\mathbb{R}
^{+} \\ 
& \left( \lambda _{n}\right) _{n\in 
\mathbb{N}
^{\ast }} & \mapsto & \underset{n=1}{\overset{+\infty }{\sum }}\left \vert
\lambda _{n}\right \vert%
\end{array}%
\end{equation*}

\item We denote, for each $n\in 
\mathbb{N}
^{\ast },$ by $e_{n}$ the element of $l^{1}\left( \mathbb{K}\right) $  defined by the
relation $e_{n}:=(\delta _{k,n})_{k\in 
\mathbb{N}
^{\ast }}$, where $\delta _{k,n}=1$ if $k=n$ and $\delta _{k,n}=0$ if $k\neq
n.$

\item We denote, for each $n\in 
\mathbb{N}
^{\ast },$ by $u_{n}^{\ast }$ the element of $\left( l^{1}\left( \mathbb{K}%
\right) \right) ^{\ast }$ defined, for each $\lambda :=\left( \lambda
_{k}\right) _{k\in 
\mathbb{N}
^{\ast }}\in l^{1}\left( \mathbb{K}\right)$ by the relation $u_{n}^{\ast }\left( \lambda \right) :=\lambda
_{n}.$

\item We denote by $l^{\infty }(\mathbb{K)}$ the $\mathbb{K}$-vector space
of sequences $\lambda :=\left( \lambda _{n}\right) _{n\in 
\mathbb{N}
^{\ast }}$ such that $\lambda _{n}\in \mathbb{K}$ for each $n\in 
\mathbb{N}
^{\ast }$ and $\sup_{n\in 
\mathbb{N}
^{\ast }}\left( \left \vert \lambda _{n}\right \vert \right) <+\infty .\ $It
is a classical result that $l^{\infty }(\mathbb{K)}$ is a Banach space for
the norm :%
\begin{equation*}
\begin{array}{cccc}
\left \Vert .\right \Vert _{l^{\infty }\left( \mathbb{K}\right) }: & 
l^{\infty }\left( \mathbb{K}\right) & \rightarrow & 
\mathbb{R}
^{+} \\ 
& \left( \mu _{n}\right) _{n\in 
\mathbb{N}
^{\ast }} & \mapsto & \underset{n\in 
\mathbb{N}
^{\ast }}{\sup }\left( \left \vert \mu _{n}\right \vert \right)%
\end{array}%
\end{equation*}

\item We denote by $\Psi $ the mapping :%
\begin{equation*}
\begin{array}{cccc}
\Psi : & l^{\infty }\left( \mathbb{K}\right) & \rightarrow & \left(
l^{1}\left( \mathbb{K}\right) \right) ^{\ast } \\ 
& \mu :=\left( \mu _{n}\right) _{n\in 
\mathbb{N}
^{\ast }} & \mapsto & \Psi \left( \mu \right)%
\end{array}%
\end{equation*}%
where :%
\begin{equation*}
\Psi \left( \mu \right) \left( \left( \lambda _{n}\right) _{n\in 
\mathbb{N}
^{\ast }}\right) :=\underset{n=1}{\overset{+\infty }{\sum }}\mu _{n}\lambda
_{n}
\end{equation*}%
It is well-known that $\Psi $ is an isometric isomorphism from $l^{\infty
}\left( \mathbb{K}\right) $ onto $l^{1}\left( \mathbb{K}\right) $ \cite[page 85, example 1.10.3]{meg}
.
\item We denot by $J_{X}$ the canonical linear mapping : 
\begin{equation*}
\begin{array}{cccc}
J_{X}: & X & \rightarrow & X^{\ast \ast } \\ 
& x & \mapsto & J_{X}(x)%
\end{array}%
\end{equation*}%
defined for each $x\in X$ and $x^{\ast }\in X^{\ast }$ by the formula $%
J_{X}(x)(x^{\ast })=x^{\ast }\left( x\right) $. It is well-known
 \cite[page 98, proposition 1.11.3]{meg} that the linear mapping $J_{X}$ is an
isometry from $X$ into $X^{\ast \ast }.$
\item Let $\left( x_{n}\right) _{n\in \mathbb{N}^{\ast }}$ be a sequence of
elements of $X.$ The series $\sum x_{n}$ of $X$ is said to be weakly
unconditionally convergent \cite[pages 58-59]{woj} if the series $\sum
\left \vert x^{\ast }\left( x_{n}\right) \right \vert $ is convergent for
each $x^{\ast }\in X^{\ast }$.

\item The Banach space $X$ is said to be weakly sequentially complete \cite
[page 218, definition 2.5.23]{meg} \cite[pages 37-38]{kal}  if for each
sequence $(x_{n})_{n\in \mathbb{N}^{\ast }}$ of $X$ such that $\underset{%
n\rightarrow +\infty }{\lim }x^{\ast }(x_{n})$ exists for every $x^{\ast
}\in X^{\ast },$ there exists $x\in X$ such that $\underset{n\rightarrow
+\infty }{\lim }x^{\ast }(x_{n})=x^{\ast }(x)$ for every $x^{\ast }\in
X^{\ast }$.
\item A sequence $\mathfrak{X}:=\left( \left( x_{n},y_{n}^{\ast }\right)
\right) _{n\in \mathbb{N}^{\ast }}\subset X\times X^{*}$ is called a paire of $X$.

\item The sequence $\mathfrak{X}^{\ast }:=\left( \left( y_{n}^{\ast
},J_{X}\left( x_{n}\right) \right) \right) _{n\in \mathbb{N}^{\ast }}\subset X^{*}\times X^{**}$ is
called the dual paire of the paire $\mathfrak{X}.$

\item The paire $\mathfrak{X}$ is called a Schauder frame (resp.
unconditional Schauder frame ) of $X$ if for all $x\in X$, the series $\sum
y_{n}^{\ast }\left( x\right) x_{n}$ is convergent (resp. unconditionally
convergent) in $X$ to $x$.

\item The paire $\mathfrak{X}$ is said to be a besselian paire of $X$ if
there exists a constant $A>0$ such that : 
\begin{equation*}
\underset{n=1}{\overset{+\infty }{\sum }}\left \vert y_{n}^{\ast }\left(
x\right) \right \vert \left \vert x^{\ast }\left( x_{n}\right) \right \vert
\leq A\left \Vert x\right \Vert _{X}\left \Vert x^{\ast }\right \Vert
_{X^{\ast }}
\end{equation*}%
for each $x\in X$ and $x^{\ast }\in X^{\ast }.$

\item The paire $\mathfrak{X}$ is said to be a besselian Schauder frame of $%
X $ if it is both a besselian paire and a Schauder frame of $X$ .

\item A Schauder frame $\mathfrak{X}$ of $X$ is said to be shrinking if the
series $\sum $ $x^{\ast }\left( x_{n}\right) y_{n}^{\ast }$ is convergent
for every $x^{\ast }\in X^{\ast }.$

\item A Schauder frame $\mathfrak{X}$ of $X$ is said to be boundedly
complete if the series $\sum x^{\ast \ast }\left( y_{n}^{\ast }\right)
x_{n} $ is convergent in $X$ for every $x^{\ast \ast }\in X^{\ast \ast }.$
\end{enumerate}
\begin{remark}
\label{bsf.exmample.l1}
For each $\lambda =\left( \lambda _{n}\right) _{n\in \mathbb{N}^{\ast }}\in
l^{1}\left( \mathbb{K}\right) ,$ it is clear that the series $\sum
u_{n}^{\ast }\left( \lambda \right) e_{n}$ is convergent and that we have $:$
\begin{equation*}
\lambda =\underset{n=1}{\overset{+\infty }{\sum }}u_{n}^{\ast }\left(
\lambda \right) e_{n}
\end{equation*}%
On the other hand let us given $\lambda =\left( \lambda _{n}\right) _{n\in 
\mathbb{N}^{\ast }}\in l^{1}\left( \mathbb{K}\right) $ and $\xi ^{\ast }\in
\left( l^{1}\left( \mathbb{K}\right) \right) ^{\ast }.$ We set $\mu :=\Psi
^{-1}\left( \xi ^{\ast }\right) \in l^{\infty }\left( \mathbb{K}\right) $.
Hence $\xi ^{\ast }=\Psi \left( \mu \right) $ and we have $:$ 
\begin{eqnarray*}
\underset{n=1}{\overset{+\infty }{\sum }}\left \vert u_{n}^{\ast }\left(
\lambda \right) \right \vert \left \vert \xi ^{\ast }\left( e_{n}\right)
\right \vert &=&\underset{n=1}{\overset{+\infty }{\sum }}\left \vert \lambda
_{n}\right \vert \left \vert \mu _{n}\right \vert \\
&\leq &\left( \underset{n=1}{\overset{+\infty }{\sum }}\left \vert \lambda
_{n}\right \vert \right) \underset{n\in \mathbb{N}^{\ast }}{\sup }\left
\vert \mu _{n}\right \vert \\
&\leq &\left \Vert \lambda \right \Vert _{l^{1}\left( \mathbb{K}\right)
}\left \Vert \mu \right \Vert _{l^{\infty }\left( \mathbb{K}\right) } \\
&\leq &\left \Vert \lambda \right \Vert _{l^{1}\left( \mathbb{K}\right)
}\left \Vert \xi ^{\ast }\right \Vert _{\left( l^{1}\left( \mathbb{K}\right)
\right) ^{\ast }}
\end{eqnarray*}%
It follows that $\left( \left( e_{n},u_{n}^{*}\right) \right) _{n\in 
\mathbb{N}^{*}}$ is a besselian Schauder frame of $l^{1}\left( \mathbb{K}\right) 
$.
\end{remark}
\begin{remark}
Assume that the paire $\mathfrak{X}$ and its dual $\mathfrak{X}^{\ast }$ are
a Schauder frames of $X$\ and $X^{\ast }$\ respectively. Then $\mathfrak{X}$
will be a shrinking Schauder frame of $X.$
\end{remark}
\begin{proof}
Since $\mathfrak{X}^{\ast }$ is a Schauder frame of $X^{\ast }$ it follows
that the series $\sum J_{X}\left( x_{n}\right) \left( x^{\ast }\right)
y_{n}^{\ast }$ $=\sum x^{\ast }\left( x_{n}\right) y_{n}^{\ast }$ is
convergent for each $x^{\ast }\in X^{\ast }.$ Consequently $\mathfrak{X}$ is
a shrinking Schauder frame of $X.$
\end{proof}
\begin{remark}
For a besselian paire $\mathfrak{X}$ of $X$, the quantity 
\begin{equation*}
\mathcal{L}_{\mathfrak{X}}:=\underset{\left( u,u^{\ast }\right) \in \mathbb{B%
}_{X}\times \mathbb{B}_{X^{\ast }}}{\sup }\left( \underset{n=1}{\overset{%
+\infty }{\sum }}\left \vert y_{n}^{\ast }\left( u\right) \right \vert \left
\vert u^{\ast }\left( x_{n}\right) \right \vert \right)
\end{equation*}%
is finite and for each $ x\in X$ and $x^{\ast }\in X^{\ast }$%
, the following inequality holds%
\begin{equation*}
\underset{n=1}{\overset{+\infty }{\sum }}\left \vert y_{n}^{\ast }\left(
x\right) \right \vert \left \vert x^{\ast }\left( x_{n}\right) \right \vert
\leq \mathcal{L}_{\mathfrak{X}}\left \Vert x\right \Vert _{X}\left \Vert
x^{\ast }\right \Vert _{X^{\ast }}
\end{equation*}
\end{remark}
For all the material on Banach spaces, one can refer to \cite{meg}, \cite
{kal}, \cite{lin01}, \cite{lin02}, \cite{woj}. In the sequel $\left(
E,\left
\Vert \cdot \right \Vert _{E}\right) $ is a given separable Banach
space, $\left( \left( a_{n},b_{n}^{\ast }\right) \right) _{n\in \mathbb{N}%
^{\ast }}$a paire of $E$ and $p\in \left] 1,+\infty \right[ $ is a given
constant and we set $p^{\ast }=\dfrac{p}{p-1}$. Finally we will index all
the series by $\mathbb{N}^{*}$.
In the sequel $\mathcal{F}:=\left( \left( a_{n},b_{n}^{\ast }\right) \right)
_{n\in \mathbb{N}^{\ast }}$ is a fixed paire of a Banach space $E.$
\section{Fundamental results}
\begin{definition}\cite{mart}.
Given a Banach space $\left( Y,\left \Vert .\right \Vert _{Y}\right) $
and a real number $0<\alpha <1$. A linear mapping $\varphi :X\rightarrow Y$ is
said to be an $\alpha $-isometry  if the following condition
holds for each $x\in X:$%
\begin{equation*}
\left( 1-\alpha \right) \left \Vert x\right \Vert _{X}\leq \left \Vert
\varphi \left( x\right) \right \Vert _{Y}\leq \left( 1+\alpha \right) \left
\Vert x\right \Vert _{X}
\end{equation*}
\end{definition}
\begin{theorem}
\label{bess-ssi-dualbess}
The paire $\mathcal{F}$ \textit{of }$E$\textit{\ is a besselian
paire of }$E$\textit{\ if and only if its dual paire} $\mathcal{F}^{\ast }$ 
\textit{is a besselian paire of }$E^{\ast }.$\textit{\ In this case we have }%
$\mathcal{L}_{\mathcal{F}}=\mathcal{L}_{\mathcal{F}^{\ast }}.$
\end{theorem}
\begin{proof}
Assume that $\mathcal{F}$ is a besselian paire of $E$. Let $x^{\ast }\in
E^{\ast }$, $x^{\ast \ast }\in E^{\ast \ast }$. We set for each $n\in 
\mathbb{N}^{\ast }$: $U=span(x^{\ast \ast })$ and $V_{n}=span(b_{1}^{\ast
},...,b_{n}^{\ast })$. It is clear that $U$ and $V_{n}$ are finite
dimensional subspaces of $E^{\ast \ast }$ and $E^{\ast }$ respectively.%
\newline
According to the principle of local reflexivity \cite[Theorem.2]{mart}, we
can find for each $\alpha >0$ an $\alpha $-isometry $T_{n}:U\longrightarrow
E $ such that :%
\begin{equation*}
\left \{ 
\begin{array}{c}
b_{j}^{\ast }\left( T_{n}(x^{\ast \ast })\right) =x^{\ast \ast }(b_{j}^{\ast
}),\text{ }j\in \left \{ 1,...n\right \} \\ 
\left \Vert T_{n}(x^{\ast \ast })\right \Vert _{E}\leq (1+\alpha )\left
\Vert x^{\ast \ast }\right \Vert _{E^{\ast \ast }}%
\end{array}%
\right.
\end{equation*}

It follows that : 
\begin{align*}
\underset{j=1}{\overset{n}{\sum }}\left \vert x^{\ast \ast }(b_{j}^{\ast
})\right \vert \left \vert J_{E}(a_{j})\left( x^{\ast }\right) \right \vert
& =\underset{j=1}{\overset{n}{\sum }}\left \vert b_{j}^{\ast }(T_{n}(x^{\ast
\ast }))\right \vert \left \vert x^{\ast }(a_{j})\right \vert \\
& \leq \mathcal{L}_{\mathcal{F}}\left \Vert T_{n}(x^{\ast \ast })\right
\Vert _{E}\left \Vert x^{\ast }\right \Vert _{E^{\ast }} \\
& \leq (1+\alpha )\mathcal{L}_{\mathcal{F}}\left \Vert x^{\ast \ast }\right
\Vert _{E^{\ast \ast }}\left \Vert x^{\ast }\right \Vert _{E^{\ast }}
\end{align*}%
Consequently, we have 
\begin{equation*}
\underset{j=1}{\overset{+\infty }{\sum }}\left \vert x^{\ast \ast
}(b_{j}^{\ast })\right \vert \left \vert J_{E}(a_{j})\left( x^{\ast }\right)
\right \vert \leq (1+\alpha )\mathcal{L}_{\mathcal{F}}\left \Vert x^{\ast
\ast }\right \Vert _{E^{\ast \ast }}\left \Vert x^{\ast }\right \Vert
_{E^{\ast }}
\end{equation*}%
for each $\alpha >0$, $x^{\ast \ast }\in E^{\ast \ast }$ and $x^{\ast }\in
E^{\ast }$. Hence : 
\begin{equation*}
\underset{j=1}{\overset{+\infty }{\sum }}\left \vert x^{\ast \ast
}(b_{j}^{\ast })\right \vert \left \vert J_{E}(a_{j})\left( x^{\ast }\right)
\right \vert \leq \mathcal{L}_{\mathcal{F}}\left \Vert x^{\ast \ast }\right
\Vert _{E^{\ast \ast }}\left \Vert x^{\ast }\right \Vert _{E^{\ast }}
\end{equation*}%
for each $x^{\ast \ast }\in E^{\ast \ast }$ and $x^{\ast }\in E^{\ast }$.
Consequently $\mathcal{F}^{\ast }$ is a besselian paire of $E^{\ast }$and we
have $\mathcal{L}_{\mathcal{F}^{\ast }}\leq \mathcal{L}_{\mathcal{F}}.$

Assume now that $\mathcal{F}^{\ast }$ is a besselian paire of $E^{\ast }$.
Let $(x,x^{\ast })\in E\times E^{\ast }$, then we have : 
\begin{align*}
\underset{j=1}{\overset{+\infty }{\sum }}\left \vert b_{j}^{\ast }(x)\right
\vert \left \vert x^{\ast }(a_{j})\right \vert & =\overset{+\infty }{%
\underset{j=1}{\sum }}\left \vert J_{E}(x)(b_{j}^{\ast })\right \vert \left
\vert J_{E}\left( a_{j}\right) (x^{\ast })\right \vert \\
& \leq \mathcal{L}_{\mathcal{F}^{\ast }}\left \Vert J_{E}(x)\right \Vert
_{E^{\ast \ast }}\left \Vert x^{\ast }\right \Vert _{E^{\ast }} \\
& \leq \mathcal{L}_{\mathcal{F}^{\ast }}\left \Vert x\right \Vert _{E}\left
\Vert y^{\ast }\right \Vert _{E^{\ast }}
\end{align*}%
Hence $\mathcal{F}$ is a besselian paire of $E$\newline
and we have $\mathcal{L}_{\mathcal{F}}\leq \mathcal{L}_{\mathcal{F}^{\ast
}}. $

We conclude that :

(i) $\mathcal{F}$ is a besselian paire of $E$ if and only $\mathcal{F}^{\ast
}$ is a besselian paire of $E^{\ast }$.

(ii) If $\mathcal{F}$ is a besselian paire of $E$ then we will have $%
\mathcal{L}_{\mathcal{F}^{\ast }}\leq \mathcal{L}_{\mathcal{F}\text{ }}$and $%
\mathcal{L}_{\mathcal{F}}\leq \mathcal{L}_{\mathcal{F}^{\ast }},$ hence $%
\mathcal{L}_{\mathcal{F}}=\mathcal{L}_{\mathcal{F}^{\ast }}.$

The proof of the theorem is then complete.\ 
\end{proof}
\begin{proposition}
\label{bess.imp.ucv}
\begin{enumerate}
\item \textit{We assume that }$E$\textit{\ is weakly sequentially complete and
that} $\mathcal{F}$ \textit{is a besselian paire of }$E$\textit{.} \textit{%
Then for each }$x\in E,$\textit{\ the series }$\sum b_{n}^{\ast }\left(
x\right) a_{n}$\textit{\ is unconditionally convergent.}

\item \textit{We assume that }$E^{\ast }$\textit{\ is weakly sequentially
complete and that} $\mathcal{F}$ \textit{is a besselian paire of }$E$\textit{%
.} \textit{Then for each }$x^{\ast }\in E^{\ast },$\textit{\ the series }$%
\sum x^{\ast }\left( a_{n}\right) b_{n}^{\ast }$\textit{\ is unconditionally
convergent.}
\end{enumerate}
\end{proposition}
\begin{proof}
\begin{enumerate}
\item Since $\mathcal{F}$ is a besselian paire of $E$, it follows that we have
for each $x\in E$ and $x^{\ast }\in E^{\ast }:$ 
\begin{align*}
\overset{+\infty }{\underset{n=1}{\sum }}\left \vert x^{\ast }\left(
b_{n}^{\ast }\left( x\right) a_{n}\right) \right \vert & =\overset{+\infty }{%
\underset{n=1}{\sum }}\left \vert b_{n}^{\ast }\left( x\right) \right \vert
\left \vert x^{\ast }\left( a_{n}\right) \right \vert \\
& \leq \mathcal{L}_{\mathcal{F}}\left \Vert x\right \Vert _{E}\left \Vert
x^{\ast }\right \Vert _{E^{\ast }} \\
& <+\infty
\end{align*}%
Hence the series $\sum b_{n}^{\ast }\left( x\right) a_{n}$ is weakly
unconditionally convergent. Hence, since $E$ is weakly sequentially
complete, it follows, thanks to Orlicz's theorem \cite{orl}; \cite[theorem of page 66]{woj},
, that the series $\sum b_{n}^{\ast }\left( x\right)
a_{n} $ is unconditionally convergent.\newline
\item Since $\mathcal{F}$ is a besselian paire of $E,$ it follows, thanks to
the theorem \ref{bess-ssi-dualbess}, that the paire $\mathcal{F}^{\ast }$ is a besselian paire of 
$E^{\ast }.$\ Hence we have for each $x^{\ast }\in E^{\ast }$ and $x^{\ast
\ast }\in E^{\ast \ast }$ : 
\begin{align*}
\overset{+\infty }{\underset{n=1}{\sum }}\left \vert x^{\ast \ast }\left(
x^{\ast }\left( a_{n}\right) b_{n}^{\ast }\right) \right \vert & =\overset{%
+\infty }{\underset{n=1}{\sum }}\left \vert x^{\ast }\left( a_{n}\right)
\right \vert \left \vert x^{\ast \ast }\left( b_{n}^{\ast }\right) \right
\vert \\
& \leq \mathcal{L}_{\mathcal{F}^{\ast }}\left \Vert x^{\ast }\right \Vert
_{E^{\ast }}\left \Vert x^{\ast \ast }\right \Vert _{E^{\ast \ast }} \\
& <+\infty
\end{align*}%
Hence for each $x^{\ast }\in E^{\ast },$ the series $\sum x^{\ast }\left(
a_{n}\right) b_{n}^{\ast }$ is weakly unconditionally convergent. Hence
since $E^{\ast }$ is weakly sequentially complete it follows thanks to
Orlicz's theorem \cite{orl}; \cite[theorem of page 66]{woj} that the
series $\sum x^{\ast }\left( a_{n}\right) b_{n}^{\ast }$ is, for each $%
x^{\ast }\in E^{\ast },$ unconditionally convergent.
\end{enumerate}
The proof of the proposition is then complete.
\end{proof}
\begin{proposition}
\begin{enumerate}
\item We assume that $E$\textit{\ is weakly
sequentially complete and that the paire} $\mathcal{F}^{\ast }$ \textit{is a
besselian Schauder frame of }$E^{\ast }$\textit{.} \textit{Then the paire} $%
\mathcal{F}$ \textit{is a besselian Schauder frame of }$\ E.$
\item We assume that the dual space $E^{*}$  is weakly
sequentially complete and that the paire $\mathcal{F}$  is a
besselian Schauder frame of  $E$. Then the paire $
\mathcal{F}^{\ast }$  is a besselian Schauder frame of  $ E^{*}$.
\end{enumerate}
\end{proposition}
\begin{proof}
\begin{enumerate}
\item Since $E$ is weakly sequentially complete and that $\mathcal{F}^{\ast }$
is a besselian Schauder frame of $E^{\ast }$ it follows, from the 
theorem \ref{bess-ssi-dualbess} and the proposition \ref{bess.imp.ucv}, 
that  $\mathcal{F}$ is a besselian paire of $E$
 and that the series $\sum b_{n}^{\ast }(x)a_{n}$ is
unconditionally convergent for each $x\in E$. Let us then
consider the mapping:
\begin{equation*}
\begin{array}{cccc}
S: & E & \rightarrow & E \\ 
& x & \mapsto & \overset{+\infty }{\underset{n=1}{\sum }}b_{n}^{\ast }\left(
x\right) a_{n}%
\end{array}%
\end{equation*}%
Then we have for each $x^{\ast }\in E^{\ast }$ and $x\in E$ : 
\begin{align*}
x^{\ast }\left( S(x)\right) & =\overset{+\infty }{\underset{n=1}{\sum }}%
x^{\ast }\left( a_{n}\right) b_{n}^{\ast }(x) \\
& =J_{E}\left( x\right) \left( \overset{+\infty }{\underset{n=1}{\sum }}%
x^{\ast }\left( a_{n}\right) b_{n}^{\ast }\right) \\
& =J_{E}\left( x\right) \left( x^{\ast }\right) \\
& =x^{\ast }\left( x\right)
\end{align*}%
Consequently, $S(x)=x$ , $x\in E$. So $\mathcal{F}$ is a besselian Schauder
frame of $E$.\newline
\item Since $E^{\ast }$ is weakly sequentially complete it follows, from the
theorem \eqref{bess-ssi-dualbess} and the proposition \ref{bess.imp.ucv},
 that $\mathcal{F}^{*}$ is a besselian paire of $E^{*}$ and that
the series $\sum J_{E}(a_{n})(x^{\ast })b_{n}^{\ast }$ is unconditionally
convergent for each $x^{\ast }\in E^{\ast }$. Let us then
consider the mapping :%
\begin{equation*}
\begin{array}{cccc}
T: & E^{\ast } & \rightarrow & E^{\ast } \\ 
& x^{\ast } & \mapsto & \overset{+\infty }{\underset{n=1}{\sum }}J_{E}\left(
a_{n}\right) \left( x^{\ast }\right) b_{n}^{\ast }%
\end{array}%
\end{equation*}%
Then we have for each $x^{\ast }\in E^{\ast }$ and $x\in E$: 
\begin{align*}
T(x^{\ast })(x)& =\overset{+\infty }{\underset{n=1}{\sum }}x^{\ast }\left(
a_{n}\right) b_{n}^{\ast }(x) \\
& =J_{E}\left( x\right) \left( \overset{+\infty }{\underset{n=1}{\sum }}%
J_{E}(a_{n})\left( x^{\ast }\right) b_{n}^{\ast }\right) \\
& =J_{E}\left( x\right) \left( x^{\ast }\right) \\
& =x^{\ast }\left( x\right)
\end{align*}%
Consequently we have, for each $x^{\ast }\in E^{\ast },$ $T(x^{\ast
})=x^{\ast }$. So $\mathcal{F}^{\ast }$ is a besselian Schauder frame of $%
E^{\ast }$.\newline
\end{enumerate}
The proof of the proposition is then complete.
\end{proof}
\begin{proposition}
\label{wsc.shrin.bound}
\begin{enumerate}
\item  Assume that  $E^{\ast }$ \textit{is weakly sequentially complete
and that }$\mathcal{F}$ \textit{is a besselian Schauder frame of }$E.$ 
\textit{Then} $\mathcal{F}$ \textit{is shrinking.}
\item   Assume that  $E$ \textit{is weakly sequentially complete and that 
}$\mathcal{F}$ \textit{is a besselian Schauder frame of }$E.$ \textit{Then} $%
\mathcal{F}$ \textit{is boundedly complete.}
\end{enumerate}
\end{proposition}
\begin{proof}
\begin{enumerate}
\item Since $E^{\ast }$\ is weakly sequentially complete and $\mathcal{F}$
is a besselian Schauder frame of $E$, it follows from  the 
proposition \ref{bess.imp.ucv} that the series 
$\sum x^{\ast }\left( a_{n}\right) b_{n}^{\ast }$\ is, for each $x^{\ast
}\in E^{\ast },$ unconditionally convergent of $E$. Hence $\mathcal{F}$ is a
shrinking Schauder frame in $E^{*}$.
\item Since $\mathcal{F}$ is a besselian paire of $E$ it follows, 
thanks to theorem \ref{bess-ssi-dualbess}, that $\mathcal{F}^{\ast }$ is a besselian paire of $E^{\ast }.$
It follows that we have for each $x^{\ast }\in E^{\ast }$ and $x^{\ast \ast
}\in E^{\ast \ast }$: 
\begin{align*}
\overset{+\infty }{\underset{n=1}{\sum }}\left \vert x^{\ast }\left( x^{\ast
\ast }\left( b_{n}^{\ast }\right) a_{n}\right) \right \vert & =\overset{%
+\infty }{\underset{n=1}{\sum }}\left \vert x^{\ast }\left( a_{n}\right)
\right \vert \left \vert x^{\ast \ast }\left( b_{n}^{\ast }\right) \right
\vert \\
& \leq \mathcal{L}_{\mathcal{F}^{\ast }}\left \Vert x^{\ast }\right \Vert
_{E^{\ast }}\left \Vert x^{\ast \ast }\right \Vert _{E^{\ast \ast }} \\
& <+\infty
\end{align*}%
It follows that the series $\sum x^{\ast \ast }\left( b_{n}^{\ast }\right)
a_{n}$ is, for each $x^{\ast \ast }\in E^{\ast \ast },$ weakly
unconditionally convergent. Hence since $E$ is weakly sequentially complete
it follows, thanks to Orlicz's theorem \cite{orl}, \cite[theorem of
page 66]{woj} that the series $\sum x^{\ast \ast }\left( b_{n}^{\ast }\right)
a_{n}$ is unconditionally convergent for each $x^{\ast \ast }\in E^{\ast
\ast }$. Consequently $\mathcal{F}$ is a boundedly complete Schauder
frame of $E$.
\end{enumerate}
The proof of the proposition is then complete.
\end{proof}
\begin{proposition}
\label{shrin.iff.dualframe}
\textit{\ Assume that} $\mathcal{F}$ is a \textit{Schauder frame} \textit{of 
}$E.$ \textit{Then} $\mathcal{F}$\  \textit{is shrinking if and only if }$
\mathcal{F}^{\ast }$\textit{\ is a Schauder frame of }$E^{\ast }$\textit{.}
\end{proposition}
\begin{proof}
Assume that $\mathcal{F}=\left( \left( a_{n},b_{n}^{\ast }\right) \right)
_{n\in \mathbb{N}^{\ast }}$ is shrinking. For each $x^{\ast }\in E^{\ast }$
and $m\in\mathbb{N}^{\ast } $ we have 
\begin{align*}
\left \Vert x^{\ast }-\overset{m}{\underset{n=1}{\sum }}J_{E}(a_{n})\left(
x^{\ast }\right) b_{n}^{\ast }\right \Vert _{E^{\ast }}& =\underset{x\in 
\mathbb{B}_{E}}{\sup }\left \vert x^{\ast }(x)-\overset{m}{\underset{n=1}{%
\sum }}x^{\ast }(a_{n})b_{n}^{\ast }(x)\right \vert \\
& =\underset{x\in \mathbb{B}_{E}}{\sup }\left \vert x^{\ast }\left( \overset{%
+\infty }{\underset{n=1}{\sum }}b_{n}^{\ast }(x)a_{n}\right) -x^{\ast
}\left( \overset{m}{\underset{n=1}{\sum }}b_{n}^{\ast }(x)a_{n})\right)
\right \vert \\
& =\underset{x\in \mathbb{B}_{E}}{\sup }\left \vert x^{\ast }\left( \overset{%
+\infty }{\underset{n=m+1}{\sum }}b_{n}^{\ast }(x)a_{n}\right) \right \vert
\\
& =\underset{x\in \mathbb{B}_{E}}{\sup }\left \vert \left( \overset{+\infty }%
{\underset{n=m+1}{\sum }}x^{\ast }(a_{n})b_{n}^{\ast }\right) (x)\right \vert
\\
& =\left \Vert \overset{+\infty }{\underset{n=m+1}{\sum }}x^{\ast
}(a_{n})b_{n}^{\ast }\right \Vert _{E^{\ast }}
\end{align*}%
It follows that the series $\sum J_{E}(a_{n})\left( x^{\ast }\right)
b_{n}^{\ast }$ is convergent to $x^{\ast }$. Consequently, $\mathcal{F}%
^{\ast }$ is a Schauder frame of $E^{\ast }$.

Assume now that $\mathcal{F}^{\ast }$ is a Schauder frame of $E^{\ast }$.
Then for each $x^{\ast }\in E^{\ast },$ the series $\sum J_{E}(a_{n})\left(
x^{\ast }\right) b_{n}^{\ast }=\sum x^{\ast }(a_{n})b_{n}^{\ast }$ is
convergent to $x^{\ast }.$ Hence $\mathcal{F}$ is a shrinking Schauder frame
of $E$.

The proof of the proposition is then achieved.
\end{proof}
\begin{definition}\cite[page 42, definition
1.6.2]{meg}.
A mapping $\omega :X\rightarrow 
\mathbb{R}^{+}$ is said to be countably subadditive  if $\omega \left( \underset{n=1}{\overset{+\infty }{\sum }}%
z_{n}\right) \leq \underset{n=1}{\overset{+\infty }{\sum }}\omega \left(
z_{n}\right) $ for each convergent series $\sum z_{n}$ in $X$.

\end{definition}
\begin{theorem}
\label{zabr}
$\mathcal{F}$ \textit{is a besselian paire of }$E$\textit{\ if and only if
the following condition holds for each }$x^{\ast }\in E^{\ast }$\textit{\
and }$x^{\ast \ast }\in E^{\ast \ast }:$%
\begin{equation}
\underset{n=1}{\overset{+\infty }{\sum }}\left \vert x^{\ast \ast }\left(
b_{n}^{\ast }\right) \right \vert \left \vert x^{\ast }\left( a_{n}\right)
\right \vert <+\infty  \label{w}
\end{equation}
\end{theorem}
\begin{proof}
Since $\mathcal{F}$ is a  besselian paire of $E$, it follows 
thanks to the theorem \ref{bess-ssi-dualbess},
 that the paire $\mathcal{F}^{\ast }$ is a besselian paire of $E^{*}$. Hence
the follwing inequality holds for each $x^{\ast }\in E^{\ast }$\textit{\ }and%
\textit{\ }$x^{\ast \ast }\in E^{\ast \ast }:$%
\begin{eqnarray*}
\underset{n=1}{\overset{+\infty }{\sum }}\left \vert x^{**}\left(
b_{n}^{\ast }\right) \right \vert \left \vert x^{*}\left( a_{n}\right) \right \vert &=&\underset{n=1}{\overset{+\infty }{%
\sum }}\left \vert x^{\ast \ast }\left( b_{n}^{\ast }\right) \right \vert
\left \vert J_{E}\left( a_{n}\right) \left( x^{\ast }\right) \right \vert \\
&\leq &\mathcal{L}_{\mathcal{F}^{\ast }}\left \Vert x^{\ast \ast }\right
\Vert _{E^{\ast \ast }}\left \Vert x^{\ast }\right \Vert _{E^{\ast }}
\end{eqnarray*}%
Consequently the condition (\ref{w}) holds for each $x^{\ast }\in E^{\ast }$
and $x^{\ast \ast }\in E^{\ast \ast }.$

Assume now that the condition (\ref{w}) holds for each $x^{\ast }\in E^{\ast
}$ and $x^{\ast \ast }\in E^{\ast \ast }$. It follows that the series $\sum
J_{E}(x)\left( b_{n}^{\ast } \right)x^{*}\left( a_{n}\right)=
\sum x^{\ast }\left( b_{n}^{\ast }\left( x\right) a_{n}\right) $ and $\sum
x^{\ast \ast }\left( x^{\ast }\left( a_{n}\right) b_{n}^{\ast }\right) $ are
unconditionally convergent for all $x\in E,$ $x^{\ast }\in E^{\ast }$ and $%
x^{**}\in E^{\ast \ast }.$ Consequently, there exists 
\cite[proposition 4, page 59 ]{woj}
for each $x\in E$ and $x^{\ast }\in E^{\ast }$ a
constants $C_{x},D_{x^{\ast }}>0$ depending respectively on $x$ and $x^{\ast
}$ such that : 
\begin{equation}
\; \overset{+\infty }{\underset{n=1}{\sum }}\left \vert b_{n}^{\ast
}(x)\right \vert \left \vert x^{\ast }(a_{n})\right \vert \leq C_{x}\left
\Vert x^{\ast }\right \Vert _{E^{\ast }},\;x^{\ast }\in E^{\ast }
\label{closed-graph1}
\end{equation}%
\begin{equation}
\overset{+\infty }{\underset{n=1}{\sum }}\left \vert x^{\ast \ast
}(b_{n}^{\ast })\right \vert \left \vert x^{\ast }(a_{n})\right \vert \leq
D_{x^{\ast }}\left \Vert x^{\ast \ast }\right \Vert _{E^{\ast \ast }}
\label{closed-graph2}
\end{equation}%
It follows from the inequality (\ref{closed-graph2}) that : 
\begin{equation*}
\overset{+\infty }{\underset{n=1}{\sum }}\left \vert J_{E}(x)(b_{n}^{\ast
})\right \vert \left \vert y^{\ast }(a_{n})\right \vert \leq D_{y^{\ast
}}\left \Vert J_{E}(x)\right \Vert _{E^{\ast \ast }},\;x\in E,\text{ }%
y^{\ast }\in E^{\ast }
\end{equation*}%
that is : 
\begin{equation}
\overset{+\infty }{\underset{n=1}{\sum }}\left \vert b_{n}^{\ast }(x)\right
\vert \left \vert y^{\ast }(a_{n})\right \vert \leq D_{y^{\ast }}\left \Vert
x\right \Vert _{E},\;x\in E,\text{ }y^{\ast }\in E^{\ast }
\label{closed-graph3}
\end{equation}%
The inequalities (\ref{closed-graph1}) and (\ref{closed-graph3})\ entail
that the following mappings :%
\begin{equation*}
\begin{array}{cccc}
f: & E & \rightarrow & 
\mathbb{R}
^{+} \\ 
& x & \mapsto & \underset{u^{\ast }\in \mathbb{B}_{E^{\ast }}}{\sup }\left( 
\overset{+\infty }{\underset{n=1}{\sum }}\left \vert b_{n}^{\ast }(x)\right
\vert \left \vert u^{\ast }(a_{n})\right \vert \right)%
\end{array}%
\end{equation*}%
\begin{equation*}
\begin{array}{cccc}
g: & E^{\ast } & \rightarrow & 
\mathbb{R}
^{+} \\ 
& y^{\ast } & \mapsto & \underset{u\in \mathbb{B}_{E}}{\sup }\left( \overset{%
+\infty }{\underset{n=1}{\sum }}\left \vert b_{n}^{\ast }(u)\right \vert
\left \vert y^{\ast }(a_{n})\right \vert \right)%
\end{array}%
\end{equation*}%
are well-defined. We prove by direct computations that $f$ and $g$ are
seminorms on $E$ and $E^{\ast }$ respectively and that we have for every$%
\;x\in E\;$ and $x^{\ast }\in E^{\ast }$: 
\begin{equation*}
\left \{ 
\begin{array}{c}
\overset{+\infty }{\underset{n=1}{\sum }}\left \vert b_{n}^{\ast }(x)\right
\vert \left \vert x^{\ast }(a_{n})\right \vert \leq f(x)\left \Vert x^{\ast
}\right \Vert _{E^{\ast }} \\ 
\overset{+\infty }{\underset{n=1}{\sum }}\left \vert b_{n}^{\ast }(x)\right
\vert \left \vert x^{\ast }(a_{n})\right \vert \leq g(x^{\ast })\left \Vert
x\right \Vert _{E}%
\end{array}%
\right.
\end{equation*}%
Let us prove now that $f$ and $g$ are both countably subadditive. Indeed,
let $\sum v_{k}$ be a convergent series in the Banach space $E$. Then we
have for each $x^{\ast }\in E^{\ast }$: 
\begin{align*}
\overset{+\infty }{\underset{n=1}{\sum }}\left \vert b_{n}^{\ast }\left( 
\overset{+\infty }{\underset{k=1}{\sum }}v_{k}\right) \right \vert \left
\vert x^{\ast }(a_{n})\right \vert & \leq \overset{+\infty }{\underset{n=1}{%
\sum }}\left( \overset{+\infty }{\underset{k=1}{\sum }}\left \vert
b_{n}^{\ast }(v_{k})\right \vert \left \vert x^{\ast }(a_{n})\right \vert
\right) \\
& \leq \overset{+\infty }{\underset{k=1}{\sum }}\left( \overset{+\infty }{%
\underset{n=1}{\sum }}\left \vert b_{n}^{\ast }(v_{k})\right \vert \left
\vert x^{\ast }(a_{n})\right \vert \right) \\
& \leq \left( \overset{+\infty }{\underset{k=1}{\sum }}f(v_{k})\right) \left
\Vert x^{\ast }\right \Vert _{E^{\ast }}
\end{align*}%
It follows that : 
\begin{equation*}
f\left( \overset{+\infty }{\underset{k=1}{\sum }}v_{k}\right) \leq \overset{%
+\infty }{\underset{k=1}{\sum }}f(v_{k})
\end{equation*}%
Hence $f$ is countably subadditive. We prove similarly that $g$ is countably
subadditive. Thanks to Zabre\u{\i}ko's lemma \cite{zbr}, 
\cite[lemma 1.6.3., page 42]{meg}, that $f$ (resp. $g$) is continuous on $E$ (resp. $E^{\ast
} $).\newline
We consider now the mapping : 
\begin{equation*}
\begin{array}{cccc}
U: & E\times E^{\ast } & \rightarrow & l^{1}(\mathbb{K}) \\ 
& (x,x^{\ast }) & \mapsto & \left( b_{n}^{\ast }(x)x^{\ast }(a_{n})\right)
_{n\in \mathbb{N}^{\ast }}%
\end{array}%
\end{equation*}%
It is clear that $U$ is well-defined since the numerical series $\sum
\left
\vert b_{n}^{\ast }(x)\right \vert \left \vert y^{\ast
}(a_{n})\right
\vert $ is convergent for each $x\in E$ and $x^{\ast }\in
E^{\ast }$. Furthermore $U$ is bilinear. Let us show that $U$ is continuous.
Indeed, let $x\in E$ and $x^{\ast }\in E^{\ast }$ and $(x_{k},x_{k}^{\ast
})_{k\in \mathbb{N}^{\ast }}$ be a sequence in $E\times E^{\ast }$ which is
convergent to $(x,x^{\ast })$. We have for every $k\in \mathbb{N}^{\ast }$ : 
\begin{align*}
\left \Vert U((x,x^{\ast }))-U((x_{k},x_{k}^{\ast }))\right \Vert _{l^{1}(%
\mathbb{K})}& =\overset{+\infty }{\underset{n=1}{\sum }}\left \vert
b_{n}^{\ast }(x)x^{\ast }(a_{n})-b_{n}^{\ast }(x_{k})x_{k}^{\ast
}(a_{n})\right \vert \\
& =\overset{+\infty }{\underset{n=1}{\sum }}\left \vert b_{n}^{\ast
}(x-x_{k})x^{\ast }(a_{n})-b_{n}^{\ast }(x_{k})(x_{k}^{\ast }-x^{\ast
})(a_{n})\right \vert \\
& \leq \overset{\infty }{\underset{n=1}{\sum }}\left \vert b_{n}^{\ast
}(x-x_{k})\right \vert \left \vert x^{\ast }(a_{n})\right \vert +\overset{%
\infty }{\underset{k=0}{\sum }}\left \vert b_{n}^{\ast }(x_{k})\right \vert
\left \vert (x_{k}^{\ast }-x^{\ast })(a_{n})\right \vert \\
& \leq f(x-x_{k})\left \Vert x^{\ast }\right \Vert _{E^{\ast
}}+g(x_{k}^{\ast }-x^{\ast })\left \Vert x_{k}\right \Vert _{E}
\end{align*}%
But $f$ is continuous on $E$ and $g$ is continuous on $E^{\ast }$. It
follows that : 
\begin{equation*}
\lim_{k\rightarrow +\infty }f(x-x_{k})=\lim_{k\rightarrow +\infty
}g(x_{k}-x^{\ast })=0\; \;
\end{equation*}%
Consequently : 
\begin{equation*}
\lim_{k\rightarrow +\infty }\left \Vert U((x,x^{\ast
}))-U((x_{k},x_{k}^{\ast }))\right \Vert _{l^{1}(\mathbb{K})}=0
\end{equation*}%
Hence the bilinear mapping $U:E\times E^{\ast }\longrightarrow l^{1}(\mathbb{%
K})$ is continuous. It follows that there exists a constant $C>0$ such that
: 
\begin{equation*}
\overset{+\infty }{\underset{n=1}{\sum }}\left \vert b_{n}^{\ast }(x)\right
\vert \left \vert x^{\ast }(a_{n})\right \vert \leq C\left \Vert x\right
\Vert _{E}\left \Vert x^{\ast }\right \Vert _{E^{\ast }}
\end{equation*}%
for every $x\in E$ and $x^{\ast }\in E^{\ast }$. It follows that $\mathcal{F}
$ is a besselian paire of $E$.

Hence the proof of the theorem is then complete.
\end{proof}
\begin{lemma}\cite{bro}.
The linear mapping :
\begin{equation*}
\begin{array}{cccc}
P_{X}: & X^{\ast \ast \ast } & \rightarrow & X^{\ast } \\ 
& u^{\ast \ast \ast } & \mapsto & u^{\ast \ast \ast }\circ J_{X}%
\end{array}%
\end{equation*}%
 is continuous from $X^{\ast \ast \ast }$ onto $X^{\ast }$.
\end{lemma}
\begin{remark}
\label{rephrased.zabr}
Using the definition of weakly unconditionally convergent series in a
Banach, the theorem \ref{zabr} can be rephrased as $:\mathcal{F}$\textit{\ is a
besselian paire of }$E$\textit{\ if and only if for each }$x^{\ast }\in
E^{\ast }$\textit{\ the series }$\sum x^{\ast }\left( a_{n}\right)
b_{n}^{\ast }$\textit{\ is weakly unconditionally convergent.}

Relying on the theorem \ref{zabr}, we obtain a new proof of
 theorem \ref{bess-ssi-dualbess} which do
not use the principle of local reflexivity.
\end{remark}
\begin{proof}
\begin{enumerate}
\item Assume that $\mathcal{F}$ is a besselian paire. Then, thanks to the
theorem \ref{zabr}, the following condition holds for each $x^{\ast \ast }\in
E^{\ast \ast }$ and $x^{\ast \ast \ast }\in E^{\ast \ast \ast }:$ 
\begin{eqnarray*}
\underset{n=1}{\overset{+\infty }{\sum }}\left \vert x^{\ast \ast }\left(
b_{n}^{\ast }\right) \right \vert \left \vert x^{\ast \ast \ast }\left(
J_{E}\left( a_{n}\right) \right) \right \vert &=&\underset{n=1}{\overset{%
+\infty }{\sum }}\left \vert P_{E}\left( x^{\ast \ast \ast }\right) \left(
a_{n}\right) \right \vert \left \vert x^{\ast \ast }\left( b_{n}^{\ast
}\right) \right \vert \\
&<&+\infty
\end{eqnarray*}%
It follows, thanks to the theorem \ref{zabr}, that $\mathcal{F}^{\ast }$ is a
besselian paire of $E^{\ast }.$
\item Assume  now that $\mathcal{F}^{\ast }$ is a besselian paire of $E^{\ast }$.
Let $x^{\ast }\in E^{\ast }$ and $x^{\ast \ast }\in E^{\ast \ast }.$ Then
there exists $y^{\ast \ast \ast }\in E^{\ast \ast \ast }$ such that $x^{\ast
}=P_{E}\left( y^{\ast \ast \ast }\right) $. It follows, by virtue of the theorem
\ref{zabr}, that :%
\begin{eqnarray*}
\underset{n=1}{\overset{+\infty }{\sum }}\left \vert y^{\ast \ast }\left(
b_{n}^{\ast }\right) \right \vert \left \vert x^{\ast }\left( a_{n}\right)
\right \vert &=&\underset{n=1}{\overset{+\infty }{\sum }}\left \vert y^{\ast
\ast }\left( b_{n}^{\ast }\right) \right \vert \left \vert P_{E}\left(
y^{\ast \ast \ast }\right) \left( a_{n}\right) \right \vert \\
&=&\underset{n=1}{\overset{+\infty }{\sum }}\left \vert y^{\ast \ast }\left(
b_{n}^{\ast }\right) \right \vert \left \vert y^{\ast \ast \ast }\left(
J_{E}\left( a_{n}\right) \right) \right \vert \\
&<&+\infty
\end{eqnarray*}%
Consequently  $\mathcal{F}$ is a besselian paire.
\end{enumerate}
The new proof of theorem \ref{bess-ssi-dualbess} is then complete.
\end{proof}
\begin{proposition}
 Assume that  $\mathcal{F}^{\ast }$\textit{\ is an unconditional
Schauder frame of }$E^{\ast },$ then $\mathcal{F}$  is a
besselian paire of  $E$.
\end{proposition}
\begin{proof}
The assumption on $\mathcal{F}^{*}$ entails that the series $\sum J_{E}\left(
a_{n}\right) (x^{\ast })b_{n}^{\ast }=\sum x^{\ast }\left( a_{n}\right)
b_{n}^{\ast }$ is weakly unconditionally convergent for each $x^{\ast }\in
E^{\ast }.$ Consequently, thanks to the remark \ref{rephrased.zabr}, $\mathcal{F}$ is a
besselian paire of $E$.
\end{proof}
\begin{lemma}\cite[page 85, example 1.10.2]{meg}.
The  mapping $ \Phi _{p}:  L_{p^{\ast }}\left( [0,1]\right)  \rightarrow  \left(
L_{p}\left( [0,1]\right) \right) ^{\ast }$
defined for each $f\in L_{p^{\ast }}\left( [0,1]\right) $ and $g\in
L_{p}\left( [0,1]\right) $ by the formula 
$\Phi_{p}(f)(g)=\int_{0}^{1}fgdx$
is an isometric isomorphism from $L_{p^{\ast }}\left(
[0,1]\right) $ onto $\left( L_{p}\left( [0,1]\right) \right) ^{\ast }$.
\end{lemma}
\begin{definition}\cite[pages: 359-361]{meg}.
  We consider the Haar system $(h_{n})_{n\in \mathbb{N}^{\ast }}$ of $
L_{p}\left( [0,1]\right) $ defined as follows. Let $h_{1}$ be $1$ on $\left[ 0,1\right[ $ and $0$ at $1$.
When $n\geq 2$, we define $h_{n}$ by letting $m$ be the positive integer
such that $2^{m-1}<n\leq 2^{m}$, then let 
\begin{equation*}
h_{n}(t)=\left \{ 
\begin{array}{c}
1\;\;\text{ if } t\in \left[ \frac{2n-2}{2^{m}}-1,\frac{2n-1}{2^{m}}-1\right[ \\ 
-1\text{ if }t\in \left[ \frac{2n-1}{2^{m}}-1,\frac{2n}{2^{m}}-1\right[ \\ 
0\text{ else}%
\end{array}%
\right.
\end{equation*}
We denote by $\mathfrak{h}\left( p\right) $ the paire 
$\left (\left (h_{n}/\Vert h_{n}\Vert_{2},
\Phi_{p} \left (h_{n}/\Vert h_{n}\Vert_{2}
  \right )\right )\right )_{n\in\mathbb{N}^{*}}$ of $L_{p}\left( \left[ 0,1\right] \right) $ and by $\mathfrak{h}%
\left( p\right) ^{\ast }$ the dual paire of the paire $\mathfrak{h}\left(
p\right) ,$ that is $\mathfrak{h}\left( p\right) ^{\ast }:=\left( \left(
\Phi _{p}\left( h_{n}/\Vert h_{n}\Vert_{2}\right) ,J_{L_{p}\left( \left[ 0,1\right] \right)
}\left( h_{n}/\Vert h_{n}\Vert_{2}\right) ,\right) \right) _{n\in \mathbb{N}^{*}}.$
\end{definition}
\begin{example}
\textit{The paire} $\mathfrak{h}\left( p\right) $ \textit{is a besselian
Schauder frame of }$L_{p}\left( \left[ 0,1\right] \right) .$
\end{example}
\begin{proof}
Since $\left( \Phi _{p}\left( h_{n}/\Vert h_{n}\Vert_{2}\right) \right) _{n\in \mathbb{N}^{\ast }}$
 is a Schauder basis of $\left( L_{p}\left( \left[ 0,1\right]
\right) \right) ^{\ast }$ which is isometrically isomorphic to $L_{p^{\ast
}}\left( \left[ 0,1\right] \right) $ \cite{pal}, \cite{marc}, it
follows that the paire   $\mathfrak{h}\left( p\right) ^{\ast
}$ is an
unconditionally Schauder frame of $\left( L_{p}(0,1)\right) ^{\ast }$. Hence
according to the proposition \ref{wsc.shrin.bound} as rephrased in the remark \ref{rephrased.zabr}, the paire $%
\mathfrak{h}\left( p\right) $ is a besselian Schauder frame of $L_{p}\left( %
\left[ 0,1\right] \right) .$
\end{proof}
\begin{definition}
Let $ p,q\in]1,+\infty [ $. The amalgam space $ (L_{p},l^{q}) $ of $ L_{p} $ 
and  $ l^{q}$ is the set of the locally $L_{p}$-integrable functions on
$\mathbb{R}$ such that:
$\sum_{m=-\infty}^{+\infty} 
\left\Vert  f_{|[m,m+1]}\right\Vert^{q}_{L_{p}([m,m+1])}<+\infty$.
\end{definition}
\begin{proposition}
The amalgam space $(L_{p},l^{q})$ is a Banach space when endowd with the norm defined by:
$\left\|f \right\|_{p,q}=\left( \sum_{m=-\infty}^{+\infty} 
\left\Vert  f_{|[m,m+1]}\right\Vert^{q}_{L_{p}([m,m+1])}\right )^{1/q}$.
It is well-known that the dual of $(L_{p},l^{q}) $ is isometrically 
isomorphic to the Banach space $(L_{p^{*}},l^{q^{*}}) $. More precisely, the following
mapping :
\begin{equation*}
\begin{array}{cccc}
\Phi_{p,q} : & (L_{p^{*}},l^{q^{*}}) & \rightarrow & (L_{p},l^{q})^{*} \\ 
& f & \mapsto & \Phi_{p,q} (f)
\end{array}
\end{equation*}
defined for each  $g\in (L_{p},l^{q})$ by the formula 
$\Phi_{p,q}(f)(g)=\sum_{m=-\infty}^{+\infty}
\int_{m}^{m+1}   f(x)g(x) dx $
is an isometric isomorphism from $(L_{p^{*}},l^{q^{*}})$
onto $ (L_{p},l^{q})^{*}$. 
\end{proposition}
  \begin{notation}
  \begin{enumerate}
  \item
For $a\in \mathbb{R}$, the operation $T_{a}$, called translation by $a$,
 is defined by
$T_{a}:  (L_{p},l^{q}) \longrightarrow (L_{p},l^{q}),\;\;T_{a}f(x):=f(x-a), x\in\mathbb{R} $.
\item
  For $m\in \mathbb{Z}$, the function $\chi_{m}: \mathbb{R}\longrightarrow \mathbb{R}$ 
is defined by:
 \[\left\{
\begin{array}{l}
\chi_{m}(x)=1\;\;if\;\; x\in[m,m+1]  \\
\chi_{m}(x)=0\;\;if\;\; x\notin[m,m+1]    \\
\end{array}
\right.\]    
\item
For each function $f:[0,1]\longrightarrow \mathbb{R}$. We define
the function  $\tilde{f}:\mathbb{R}\longrightarrow \mathbb{R}$ by:
 \[\tilde{f}(x)=\left\{
\begin{array}{l}
f(x)\;\;if\;\; x\in [0,1]  \\
\;\;0\;\;\;\;if\;\; x\notin [0,1]   \\
\end{array}
\right.\]  
  \end{enumerate}
  \end{notation}

\begin{example} 
Let $1<p,q<+\infty$ be given  and $\mathcal{F}=\left( \left( a_{n},\Phi_{p}(b_{n})\right) \right) _{n\in \mathbb{N}^{*}}$
 is an unconditional  besselian Schauder frame of $L_{p}[0,1]$.  Then
$ \left (\left (T_{m}\tilde{a}_{n},\Phi_{p,q}\left (T_{m}\tilde{b}_{n}\right )\right )\right )_{(m,n)\in\mathbb{Z}\times \mathbb{N}^{*}}$
  is an unconditional besselian Schauder frame of $  (L_{p},l^{q}) $.
\end{example} 
\textbf{Proof}:

Let $m\in \mathbb{Z}$ and $ f\in(L_{p},l^{q})$. We have:
\begin{align*}
T_{-m}\left(\chi_{m}f\right )&=\overset{+\infty }{\underset{n=1}{\sum }}\left(\int_{0}^{1}\tilde{b}_{n}T_{-m}\left(\chi_{m}f\right )\right )\tilde{a}_{n} \\
 				   	 &=\overset{+\infty }{\underset{n=1}{\sum }}\left(\int_{m}^{m+1}\left (T_{m}\tilde{b}_{n}\right )\left(\chi_{m}f\right )\right )\tilde{a}_{n}\\
 				   	 &=\overset{+\infty }{\underset{n=1}{\sum }}\left(\int_{\mathbb{R}}\left (T_{m}\tilde{b}_{n}\right )f\right )\tilde{a}_{n}\\
\end{align*} 
Then 
\[\chi_{m}f=\overset{+\infty }{\underset{n=1}{\sum }}\Phi_{p,q}\left (T_{m}\tilde{b}_{n}\right )(f)T_{m}\tilde{a}_{n} \]

For each $N\in \mathbb{N}$ and $M\in \mathbb{N}$ the function
$\overset{M}{\underset{m=-N}{\sum }}\chi_{m}f  $ is an element
of $(L_{p},l^{q})$  and we have:
\begin{align*}
 \left \Vert f-\overset{M}{\underset{m=-N}{\sum }}\chi_{m}f\right\Vert^{q}_{(L_{p},l^{q})} 
 &=\overset{+\infty }{\underset{k=-\infty}{\sum }}
 \left (\int_{k}^{k+1}\left \vert f-\overset{N}{\underset{m=-N}{\sum }}\chi_{m}f\right\vert^{p}dx\right )^{q/p}\\ 			
  &=\overset{-N-1}{\underset{k= -\infty}{\sum }}
 \left (\int_{k}^{k+1}\left \vert f\right\vert^{p}dx\right )^{q/p}
 +\overset{+\infty}{\underset{k=M+1 }{\sum }}
 \left (\int_{k}^{k+1}\left \vert f\right\vert^{p}dx\right )^{q/p}
 				   	 \end{align*}
 But we have:
$$\underset{min(M,N)\rightarrow +\infty }{\lim}
\overset{-N-1}{\underset{k= -\infty}{\sum }}
 \left (\int_{k}^{k+1}\left \vert f\right\vert^{p}dx\right )^{q/p}
 +\overset{+\infty}{\underset{k=M+1 }{\sum }}
 \left (\int_{k}^{k+1}\left \vert f\right\vert^{p}dx\right )^{q/p}=0$$
It follows that:
\begin{align*}
 				 	f&=\overset{+\infty }{\underset{m=-\infty}{\sum }}\chi_{m}f \\
 				   	 	 &=\overset{+\infty }{\underset{m=-\infty}{\sum }}
 				     	 \left (\overset{\infty}{\underset{n=1}{\sum}}\Phi_{p,q}\left (T_{m}\tilde{b}_{n}\right )(f )T_{m}\tilde{a}_{n} \right )
\end{align*} 

Let $ g\in(L_{p^{*}},l^{q^{*}})$. Then by the 
H\"older inequality we obtain:
\begin{align*}
& \overset{+\infty }{\underset{m=-\infty}{\sum }}
\overset{+\infty }{\underset{n=1}{\sum }}
  \left \vert \Phi_{p,q}\left (T_{m}\tilde{b}_{n}\right )\left( f\right)\right  \vert  \left \vert  \Phi_{p,q}(g)\left( T_{m}\tilde{a}_{n}\right) \right \vert \\
&=\overset{+\infty }{\underset{m=-\infty}{\sum }}
\overset{+\infty }{\underset{n=1}{\sum }}
  \left \vert \int_{\mathbb{R}} T_{m}\left (\tilde{b}_{n}\right )f \right  \vert  
\left \vert \int_{\mathbb{R}} T_{m}\left (\tilde{a}_{n}\right )g \right  \vert\\ 
&=\overset{+\infty }{\underset{m=-\infty}{\sum }}
\overset{+\infty }{\underset{n=1}{\sum }}
  \left \vert \int_{m}^{m+1}T_{m}\left (\tilde{b}_{n}\right )f \right  \vert  
\left \vert \int_{m}^{m+1}T_{m}\left (\tilde{a}_{n}\right )g \right  \vert\\ 
&=\overset{+\infty }{\underset{m=-\infty}{\sum }}
\overset{+\infty }{\underset{n=1}{\sum }}
  \left \vert \int_{0}^{1}b_{n}T_{-m}\left (\chi_{m}f\right ) \right  \vert  
\left \vert \int_{0}^{1}a_{n}T_{-m}\left (\chi_{m}g\right ) \right  \vert\\ 
 				   	 &\leq  \mathcal{L}_{\mathcal{F}}
 				   	 \overset{+\infty }{\underset{m=-\infty}{\sum }}
 				   	 \left (\int_{0}^{1}\left \vert T_{-m}\left (\chi_{m}f\right ) \right  \vert^{p}\right )^{1/p}
 				   	 \left (\int_{0}^{1}\left \vert T_{-m}\left (\chi_{m}g\right ) \right  \vert^{p^{*}}\right )^{1/p^{*}}\\
  					&\leq  \mathcal{L}_{\mathcal{F}}
 				   	 \overset{+\infty }{\underset{m=-\infty}{\sum }}
 				   	 \left (\int_{m}^{m+1}\left \vert f\right  \vert^{p}\right )^{1/p}
 				   	 \left (\int_{m}^{m+1}\left \vert g\right  \vert^{p^{*}}\right )^{1/p^{*}}\\
&=\mathcal{L}_{\mathcal{F}} \left\|f \right\|_{\left (L_{p},l^{q}\right )} \left\|\Phi_{p,q}(g) \right\|_{\left (L_{p},l^{q}\right )^{*}}
\end{align*}
Finally, since $  (L_{p},l^{q}) $ is reflexive then it is also weakly
sequentially complete. It follows, according to the proposition \eqref{bess.imp.ucv} 
 that $ \left (\left (T_{m}\tilde{a}_{n},\Phi_{p,q}\left (T_{m}\tilde{b}_{n}\right )\right )\right )_{(m,n)\in\mathbb{Z}\times \mathbb{N}^{*}}$
  is an unconditional besselian Schauder frame of $  (L_{p},l^{q}) $.\\
  \qed

The following result is a generalisation to the well known James's theorem
\cite{jam} which characterizes reflexive Banach spaces.
\begin{theorem}
\label{James}
\textit{Assume that} $\mathcal{F}$\ is a besselian Schauder frame of 
$E$. Then $E$ is reflexive if and only if $\mathcal{F}$ 
 is shrinking and boundedly complete.
\end{theorem}
\begin{proof}
Assume that $E$ is reflexive. Then $E^{\ast }$ is also reflexive \cite[Corollary 1.11.17 page 104]{meg}.
Consequently, $E$ and $E^{\ast }$ are weakly sequentially complete Banach
spaces. Since $\mathcal{F}$ is a besselian paire of $E$, it follows, thanks
to the theorem \ref{bess-ssi-dualbess}, that $\mathcal{F}^{\ast }$ is a besselian paire of $%
E^{\ast }$. It follows from proposition \ref{wsc.shrin.bound}, that $\mathcal{F}$ is shrinking
and boundedly complete.\newline
Assume now that $\mathcal{F}$ is shrinking and boundedly complete. Let be
given $x^{\ast \ast }\in E^{\ast \ast }$. Since $\mathcal{F}$ is boundedly
complete then the series $\sum x^{\ast \ast }\left( b_{n}^{\ast }\right)
a_{n}$ convergent to an element $x$ in $E$, that is: 
\begin{equation*}
x=\overset{+\infty }{\underset{n=1}{\sum }}x^{\ast \ast }(b_{n}^{\ast })a_{n}
\end{equation*}%
But $\mathcal{F}$ is shrinking. Hence, thanks to the proposition \ref{shrin.iff.dualframe}, $%
\mathcal{F}^{\ast }$ is a Schauder frame of $E^{\ast }$. It follows that : 
\begin{equation*}
y^{\ast }=\overset{+\infty }{\underset{n=1}{\sum }}y^{\ast
}(a_{n})b_{n}^{\ast },\;y^{\ast }\in E^{\ast }
\end{equation*}%
Consequently, we have for each $x^{\ast }\in E^{\ast }$: 
\begin{align*}
J_{E}(x)(x^{\ast })& =x^{*}(x)\\
& = \overset{+\infty }{\underset{n=1}{%
\sum }}x^{\ast }(a_{n})x^{**}\left(b_{n}^{\ast }\right) \\
& =x^{\ast \ast }\left( \overset{+\infty }{\underset{n=1}{%
\sum }}x^{\ast }(a_{n})b_{n}^{\ast }\right) \\
& =x^{\ast \ast }(x^{\ast })
\end{align*}%
It follows that $x^{**}=J_{E}(x)$. Thus $J_{E}$ is surjective.
Consequently $E$ is reflexive.\newline

The proof of the theorem is then complete.
\end{proof}
\begin{theorem}
\textit{Assume that }$\mathcal{F}$\textit{\ is a besselian Schauder frame of 
}$E$\textit{. Then }$E$\textit{\ is reflexive if and only if the spaces }$E$%
\textit{\ and }$E^{\ast }$\textit{\ are both weakly sequentially complete.}
\end{theorem}
\begin{proof}
Assume that $E$ is reflexive. Then $E$ and $E^{\ast }$ are both reflexive.
Hence $E$ and $E^{\ast }$ are both weakly sequentially complete.\newline
Assume now that $E$ and $E^{\ast }$ are both weakly sequentially complete.
Since $\mathcal{F}$ is a besselian Schauder frame of $E$, it follows,
according to the proposition \ref{wsc.shrin.bound}, that $\mathcal{F}$ is a besselian Schauder
frame of $E$ which is shrinking and boundedly complete. Consequently the
theorem \ref{James} entails that the Banach space $E$ is reflexive.

The proof of the theorem is then complete.
\end{proof}
\begin{definition}\cite[page 220,
definition 2.5.25]{meg} \cite[page 37, definition 2.3.4]{kal}  \cite{sch}.
$X$ is said to be a Schur space  if it satisfies the following condition : Whenever $(x_{n})_{n\in \mathbb{N}^{\ast
}}$ a sequence of $X$ and \ $x\in X$ such that $\underset{n\rightarrow
+\infty }{\lim }x^{\ast }(x_{n})=x^{*}(x)$   for every $x^{\ast }\in X^{\ast },$
then $\underset{n\rightarrow +\infty }{\lim }x_{n}=x.$
\end{definition}
\begin{corollary}
 Assume that  $E$  is an infinitely dimensional Schur space
which has a besselian Schauder frame. Then the dual space  $E^{\ast }$ 
 is not weakly sequentially complete.
\end{corollary}
\begin{proof}
Since $E$ is an infinitely dimensional Schur space, then $E$ is
not reflexive \cite[corollary 2.3.8 page 37]{kal}. \ But it is also
assumed that $E$ has a besselian Schauder frame. Consequently $E^{\ast }$ is
not weakly sequentially complete.
\end{proof}
\begin{corollary}
 The Banach space  $l^{\infty }\left( \mathbb{K}\right) $  is
not weakly sequentially complete. 
\end{corollary}
\begin{proof}
It is well-known that the space $l^{1}\left( \mathbb{K}\right) $ is a Schur
space \cite[example 2.5.24 pages 218-220]{meg}, \cite[page 37, theorem 2.3.6 ]{kal},
\cite{sch} which is infinite dimensional and for which the
paire $\left( \left( e_{n},u_{n}^{\ast }\right) \right) _{n\in 
\mathbb{N}
^{\ast }}$ is a besselian Schauder frame as proved in the remark \ref{bsf.exmample.l1}. Hence
the dual space $\left( l^{1}\left( \mathbb{K}\right) \right) ^{\ast }$ is
not weakly sequentially complete. But since $\left( l^{1}\left( \mathbb{K}%
\right) \right) ^{\ast }$ is isometrically isomorphic to the Banach space $%
l^{\infty }\left( \mathbb{K}\right) ,$ it follows that the Banach space $%
l^{\infty }\left( \mathbb{K}\right) $ is not weakly sequentially complete.
\end{proof}

\bigskip

\bigskip

\end{document}